\documentclass[a4paper]{amsart}
\usepackage{amsmath,amsthm,amssymb}
\usepackage[truedimen,margin=25truemm]{geometry}
\usepackage{url}
\usepackage[all]{xy}

\numberwithin{equation}{section}

\theoremstyle{definition}
\newtheorem{thm}{Theorem}[section]

\newtheorem{definition}[thm]{Definition}
\newtheorem{lem}[thm]{Lemma}
\newtheorem{cor}[thm]{Corollary}

\newcommand{\Mat}{{\rm Mat}}

\newcommand{\diag}{{\rm diag}}

\begin{document}
\title[]{Generalization of Frobenius' theorem for group determinants}
\author[N. Yamaguchi]{Naoya Yamaguchi}
\date{\today}

\maketitle

\begin{abstract}
Frobenius built a representation theory of finite groups in the process of obtaining the irreducible factorization of the group determinant. 
Here, we give a generalization of Frobenius' theorem. 
The generalization leads to a corollary on irreducible representations of finite groups. 
\end{abstract}

\section{\bf{Introduction}}
In this paper, we give a generalization of Frobenius' theorem. 
The generalization leads to a corollary on irreducible representations of finite groups. 
Let $G$ be a finite group, 
let $\widehat{G}$ a complete set of irreducible representations of $G$ over $\mathbb{C}$, 
and let $\mathbb{C}[x_{g}] = \mathbb{C}[x_{g} ; g \in G]$ the polynomial ring in $\{ x_{g} \: \vert \: g \in G \}$ with coefficients in $\mathbb{C}$. 
The group determinant $\Theta(G) \in \mathbb{C}[x_{g}]$ is the determinant of a matrix whose elements are independent variables $x_{g}$ corresponding to $g \in G$. 
Frobenius proved the following theorem about the irreducible factorization of the group determinant.

\begin{thm}[Frobenius' theorem \cite{conrad1998origin}]\label{thm:1.1.1}
Let $G$ be a finite group. 
Then we have the irreducible factorization: 
$$
\Theta(G) = \prod_{\varphi \in \widehat{G}} \det{\left( \sum_{g \in G} \varphi (g) x_{g} \right)^{\deg{\varphi}}}. 
$$
\end{thm}

Frobenius built a representation theory of finite groups in the process of obtaining Theorem~$\ref{thm:1.1.1}$ (see, \cite{Frobenius1968gruppen}, \cite{Frobenius1968gruppencharaktere}). 
This theorem is a generalization of Dedekind's theorem. 
\begin{thm}[Dedekind's theorem \cite{conrad1998origin}]\label{thm:1.1.2}
Let $G$ be a finite abelian group and let $\widehat{G}$ the group of characters of $G$. 
Then we have 
$$
\Theta(G) = \prod_{\chi \in \widehat{G}} \sum_{g \in G} \chi (g) x_{g}. 
$$
\end{thm}

We denote the index of a subgroup $H$ in $G$ as $[G : H]$. 
The paper~\cite{Yamaguchi2017} give a generalization of Dedekind's theorem. 

\begin{thm}[Generalization of Dedekind's theorem \cite{Yamaguchi2017}]\label{thm:1.1.4}
Let $G$ be a finite abelian group and let $H$ a subgroup of $G$. 
For every $h \in H$, there exists a homogeneous polynomial $A_{h} \in \mathbb{C}[x_{g}]$ such that $\deg{A_{h}} = \left[G : H \right]$ and 
$$
\Theta(G) = \prod_{\chi \in \widehat{H}} \sum_{h \in H} \chi(h) A_{h}. 
$$
If $H = G$, we can take $A_{h} = x_{h}$ for each $h \in H$. 
\end{thm}

We denote the group algebra of $H$ over $\mathbb{C}$ as $\mathbb{C} H$, and let $\mathbb{C}[x_{g}] H := \mathbb{C} [x_{g}] \otimes \mathbb{C} H$ for any subgroup $H$ of $G$. 
Here, we give a generalization of Theorems~$\ref{thm:1.1.1}$ and $\ref{thm:1.1.4}$. 
The theorem is as follows.

\begin{thm}[Generalization of Frobenius' theorem]\label{thm:1.1.2}
Let $G$ be a finite group, 
let $H$ a subgroup of $G$, 
let $L$ a left regular representation from $\mathbb{C}[x_{g}] G$ to $\Mat([G : H], \mathbb{C}[x_{g}] H)$, 
let $\alpha = \sum_{g \in G} x_{g} g \in \mathbb{C}[x_{g}] G$, 
and let $L(\alpha) = \sum_{h \in H} h c_{h}$, 
where $c_{h} \in \Mat([G : H], \mathbb{C}[x_{g}] \{ 1_{G} \})$. 
Then, we have 
\begin{align*}
&\Theta(G) = \prod_{\psi \in \widehat{H}} \det{\left( \sum_{h \in H} \psi(h) \otimes c_{h} \right)^{\deg{\psi}}}, 
\end{align*}
where $1_{G}$ is the unit element of $G$ and $\otimes$ is the Kronecker product. 
\end{thm}

Theorem~$\ref{thm:1.1.1}$ is the special case of $H = G$ in the generalization of Frobenius' theorem. 
Theorem~$\ref{thm:1.1.2}$ leads to the following corollary (remark that the following corollary follows from Frobenius reciprocity\cite{kondo2011group}). 

\begin{cor}\label{cor:1.1.3}
Let $G$ be a finite group and let $H$ a subgroup of $G$.
For all $\varphi \in \widehat{G}$, we have 
\begin{align*}
\deg{\varphi} \leq [G:H] \times \max\left\{ \deg{\psi} \: \vert \: \psi \in \widehat{H} \right\}. 
\end{align*}
\end{cor}

\section{\bf{Group determinant}}

Let $G$ be a finite group, 
let $L_{G}$ be the left regular representation of $G$, 
let $\{ x_{g} \: \vert \: g \in G \}$ be independent commuting variables, 
and let $\mathbb{C}[x_{g}] = \mathbb{C}[x_{g} ; g \in G]$ be the polynomial ring in $\{ x_{g} \: \vert \: g \in G \}$ with coefficients in $\mathbb{C}$. 
The group determinant is defined as follows. 

\begin{definition}[Group determinant]
The group determinant $\Theta(G)$ of $G$ is given by 
\begin{align*}
\Theta(G) := \det{\left( \sum_{g \in G} L_{G}(g) x_{g} \right)} \in \mathbb{C} [x_{g}]. 
\end{align*}
\end{definition}
That is, the group determinant $\Theta(G)$ is the determinant of the $|G| \times |G|$ matrix $\left( x_{g, h} \right)_{g, h \in G}$, 
where $x_{g, h} = x_{g h^{-1}}$ for $g, h \in G$, 
and it is thus a homogeneous polynomial of degree $|G|$ in $x_{g}$ (see, e.g.,~\cite[p.~366]{conrad1998origin}, \cite[p.~38]{Frobenius1968gruppen}, \cite[p.~142]{Hawkins1971}, \cite[p.~299]{johnson1991}, \cite[p.~224]{van2013history}). 
Frobenius proved the following theorem about the factorization of the group determinant.

\begin{thm}[Section~$1$, Theorem~$\ref{thm:1.1.1}$]\label{thm:2.1.1}
Let $G$ be a finite group, 
for which we have the irreducible factorization, 
$$
\Theta(G) = \prod_{\varphi \in \widehat{G}} \det{\left( \sum_{g \in G} \varphi (g) x_{g} \right)^{\deg{\varphi}}}. 
$$
\end{thm}

The above equation holds from the following theorem.

\begin{thm}[{\cite[Theorem~4.4.4]{benjamin2013}}]\label{thm:2.2}
Let $G$ be a finite group, 
$\left\{ \varphi_{1}, \varphi_{2}, \ldots, \varphi_{s} \right\}$ a complete set of inequivalent irreducible representations of $G$, 
$d_{i} = \deg{\varphi_{i}}$, 
and $L_{G}$ the left regular representation of $G$. 
Then, 
$$
L_{G} \sim d_{1} \varphi_{1} \oplus d_{2} \varphi_{2} \oplus \cdots \oplus d_{s} \varphi_{s}. 
$$
\end{thm}

\section{\bf{Generalization of Frobenius' theorem}}
Here, we review the left regular representation and describe some of the properties of the left regular representation that will be needed later. 
Let $B$ be a ring and let $A$ be a ring that is a free right $B$-module with an ordered basis $e = \left( e_{1} \: e_{2} \: \cdots \: e_{m} \right)$. 
In other words, $A = \oplus_{i = 1}^{m} e_{i} B$, and $B$ is a subring of $A$. 
Then, for all $a \in A$, there exists a unique $(b_{i j})_{1 \leq i, j \leq m} \in \Mat(m, B)$ such that 
$$
a e_{j} = \sum_{i = 1}^{m} e_{i} b_{i j}. 
$$
The injective $Z(A) \cap B$-algebra homomorphism $L_{e} : A \ni a \mapsto L_{e}(a) = (b_{i j})_{1 \leq i, j \leq m} \in \Mat(m, B)$ is called the left regular representation from $A$ to $\Mat(m, B)$ with respect to $e$, 
where $Z(A)$ is the center of $A$.

Let $C$ be a ring. 
In the following, 
we assume that $B$ is a free right $C$-module with an ordered basis $f = (f_{1} \: f_{2} \: \cdots \: f_{n})$. 
Then, $A$ is a free right $C$-module with an ordered basis $e \otimes f$, 
where $\otimes$ is the Kronecker product. 
We write the left regular representation from $B$ to $\Mat(n, C)$ with respect to $f$ as $L_{f}$ and that from $\Mat(m, B)$ to $\Mat(n, \Mat(m, C))$ with respect to $f \otimes I_{m}$ as $L_{f \otimes I_{m}}$. 
In terms of the regular representations, the following lemma holds. 

\begin{lem}[{\cite[Lemma~$3.6$]{doi:10.1080/25742558.2019.1683131}}]\label{lem:3.1}
The following diagram is commutative: 
\[
\xymatrix{
A \ar[r]^-{L_{e}} \ar[d]_-{L_{e \otimes f}} & \Mat(m, B) \ar[d]^-{L_{f \otimes I_{m}}} \\ 
\Mat(mn, C) \ar@{^{(}-_>}[r]^-{} & \Mat(n, \Mat(m, C)) \\ 
}
\]
That is, $L_{e \otimes f} = L_{f \otimes I_{m}} \circ L_{e}$ holds, 
where we regard $\Mat(n, \Mat(m, C))$ as $\Mat(mn, C)$. 
\end{lem} 

\begin{definition}
Let $R'$ be a ring and let $R$ be a ring that is a free right $R'$-module an ordered basis $v = (v_{1} \: v_{2} \: \cdots \: v_{l})$. 
We define $p_{v} \colon R \to R'$ and $p_{v}^{k} \colon \Mat(k, R) \to \Mat(k, R')$ by 
\begin{align*}
p_{v} \left( \sum_{i = 1}^{l} v_{i} a_{i} \right) = \sum_{i = 1}^{l} a_{i}, 
\quad p_{v}^{k} \left( (a_{i j})_{1 \leq i, j \leq k} \right) = ( p_{v} (a_{i j}) )_{1 \leq i, j \leq k}, 
\end{align*}
respectively. 
In addition, we denote $\Theta( R : R' ) \colon R \to \Mat(l, R')$ as 
$$
\Theta( R : R' ) := \det{} \circ p_{v}^{l} \circ L_{v}, 
$$
where $L_{v}$ is the left regular representation from $R$ to $\Mat(l, R')$ with respect to $v$. 
\end{definition}

We note that $L_{v}$ is covariant under a change of basis. 
However, $\Theta( R : R' )$ is invariant under a change of basis. 
From  Lemma~$\ref{lem:3.1}$, we have the following corollary. 

\begin{cor}
We have $\Theta(A : C) = \Theta(\Mat(m, B), C) \circ L_{e}$. 
\end{cor}

Let $f C := \left\{ f_{i} C \mid 1 \leq i \leq n \right\}$ and we define a product of $f_{i} C$ and $f_{j} C$ as $f_{i} C * f_{j} C := \left\{ f_{i} f_{j} c \mid c \in C \right\}$. 
Assume that the basis $e$ satisfies the following conditions: 
\begin{enumerate}
\item for any $f_{i}$ and $f_{j}$, $f_{i} C * f_{j} C \in f C$; 
\item there exists $f_{k}$ such that $f_{k} C = C$ as set; 
\item for any $f_{i}$, there exists $f_{j}$ such that $f_{i} C * f_{j} C = C$. 
\end{enumerate}
From these conditions, $f_{i}$ is invertible in $B$ and $(f C, *)$ is a group. 
For a group $G$, we denote the regular representation of $G$ as $L_{G}$, 
and we write $\diag( a_{1}, a_{2}, \ldots, a_{k})$ for a diagonal matrix whose diagonal entries starting in the upper left corner are $a_{1}, a_{2}, \ldots, a_{n}$. 
Then we have an expression for regular representations. 

\begin{cor}[{\cite[Corollary~$8.3$]{doi:10.1080/25742558.2019.1683131}}]\label{cor:3.1}
Let $\beta = \sum_{k = 1}^{n} f_{k} c_{k} \in B$, where $c_{k} \in C$. 
Then we have 
$$
L_{f}(\beta) = \diag( f_{1}^{-1}, f_{2}^{-1}, \ldots, f_{n}^{-1} ) \left( \sum_{k = 1}^{n} L_{f C}(f_{k} C) \otimes f_{k} c_{k} \right) \diag( f_{1}, f_{2}, \ldots, f_{n} ). 
$$
\end{cor}

From $\Mat(m, B)$ is a ring that is a free right $\Mat(m, C)$-module with an ordered basis $f \otimes I_{m}$, 
Corollary~$\ref{cor:3.1}$, 
and $(f \otimes I_{m}) C := \left\{ ( f_{i} \otimes I_{m} ) C \mid 1 \leq i \leq n \right\} \cong f C$ as group, 
we have the following corollary. 

\begin{cor}\label{cor:3.2}
Let $\beta = \sum_{k = 1}^{n} (f_{k} \otimes I_{m}) c_{k} \in \Mat(m, B)$, where $c_{k} \in \Mat(m, C)$. 
Then we have 
$$
L_{f \otimes I_{m}}(\beta)  = \diag( f_{1}^{-1} \otimes I_{m}, \ldots, f_{n}^{-1} \otimes I_{m} ) \left( \sum_{k = 1}^{n} L_{f C} \left( f_{k} C \right) \otimes ( f_{k} \otimes I_{m} ) c_{k} \right) \diag( f_{1} \otimes I_{m}, \ldots, f_{n} \otimes I_{m} ). 
$$
\end{cor}

From Theorem~$\ref{thm:2.2}$ and Corollary~$\ref{cor:3.2}$, 
we have the following lemma. 

\begin{lem}\label{lem:3.6}
Let $C$ be a commutative and $\beta = \sum_{k = 1}^{n} (f_{k} \otimes I_{m}) c_{k} \in \Mat(m, B)$, where $c_{k} \in \Mat(m, C)$. 
Then we have 
\begin{align*}
\Theta(\Mat(m, B), C) = \left( \det{} \circ p_{f \otimes I_{m}}^{m n} \circ L_{f \otimes I_{m}} \right)(\beta) = \prod_{\psi \in \widehat{f C}} \det{\left( \sum_{k = 1}^{n} \psi(f_{k} C) \otimes c_{k} \right)}^{\deg{\psi}}. 
\end{align*}
\end{lem}

Let $K = \{ 1_{G} \} \subset H \subset G$ be a sequence of groups, where $1_{G}$ is the unit element of $G$. 
We take a complete set $\{ e_{1}, e_{2}, \ldots, e_{m} \}$ of left coset representatives of $H$ in $G$ and $H = \{ f_{1}, f_{2}, \ldots, f_{n} \}$ of left coset representatives of $K$ in $H$. 
In the following, we assume that $A := \mathbb{C}[x_{g}] \otimes \mathbb{C} G$, $B := \mathbb{C}[x_{g}] \otimes \mathbb{C} H$ and $C := \mathbb{C}[x_{g}] \otimes \mathbb{C} K$. 
Then, $B$ is a ring, 
$A$ is a ring that is a free right $B$-module with an ordered basis $e := ( e_{1} \: e_{2} \: \cdots \: e_{m} )$, 
$C$ is a ring, 
and $B$ is a ring that is a free right $C$-module with an ordered basis $f := ( f_{1} \: f_{2} \: \cdots \: f_{n} )$. 
Let $\alpha = \sum_{g \in G} x_{g} g$. 
It is easy to show that $\Theta(A : C)(\alpha) = \Theta(G)$.  
Therefore, we have $\Theta(G) = \Theta(\Mat(m, B) : C) \circ L_{e})(\alpha)$. 
From Lemma~$\ref{lem:3.6}$, we have the following. 

\begin{thm}[Section~1, Theorem~\ref{thm:1.1.2}]\label{thm:4.1}
Let $L_{e}(\alpha) = \sum_{h \in H} (h \otimes I_{m}) c_{h} \in \Mat(m, B)$, where $c_{k} \in \Mat(m, \mathbb{C}[x_{g}])$. 
Then we have 
\begin{align*}
\Theta(G) = \prod_{\psi \in \widehat{H}} \det{\left( \sum_{h \in H} \psi(h) \otimes c_{h} \right)}^{\deg{\psi}}. 
\end{align*}
\end{thm}

The polynomial ring $\mathbb{C}[x_{g}]$ is a unique factorization domain. 
Therefore, we have the following corollary from Theorems~$\ref{thm:2.1.1}$ and $\ref{thm:4.1}$.

\begin{cor}[Section~$1$, Corollary~$\ref{cor:1.1.3}$]\label{cor:4.1.4}
Let $G$ be a finite group and let $H$ a subgroup of $G$. 
For all $\varphi \in \widehat{G}$, 
we have 
\begin{align*}
\deg{\varphi} \leq [G:H] \times \max\left\{ \deg{\psi} \: \vert \: \psi \in \widehat{H} \right\}. 
\end{align*}
\end{cor}
\begin{proof}
We have 
\begin{align*}
\deg{\varphi} 
&= \deg{ \left( \det{ \left( \sum_{g \in G} \varphi(g) x_{g} \right) } \right) } \\ 
&\leq \max \left\{ \deg{\left( \det{ \left( \sum_{h \in H} \psi(h) \otimes c_{h} \right) } \right)} \: \vert \: \psi \in \widehat{H} \right\} \\ 
&= \max \left\{ \deg{\psi} \times [G:H] \: \vert \: \psi \in \widehat{H} \right\} \\ 
&= [G:H] \times \max\left\{ \deg{\psi} \: \vert \: \psi \in \widehat{H} \right\}. 
\end{align*}
This completes the proof. 
\end{proof}

Remark that Corollary~$\ref{cor:4.1.4}$ also follows from Frobenius reciprocity\cite{kondo2011group}.

\clearpage

\thanks{\bf{Acknowledgments}}
I am deeply grateful to Prof. Hiroyuki Ochiai, Prof. Geoffrey Robinson, Dr. Benjamin Sambale, and Dr. Yuka Yamaguchi who provided helpful comments and suggestions. 
This work was supported by a grant from the Japan Society for the Promotion of Science (JSPS KAKENHI Grant Number 15J06842).

\bibliographystyle{amsplain}
\bibliography{reference}

\medskip
\begin{flushleft}
Naoya Yamaguchi \\
Faculty of Education \\
University of Miyazaki  \\
1-1 Gakuen Kibanadai-nishi \\
Miyazaki, 889-2192 \\
Japan\\
n-yamaguchi@cc.miyazaki-u.ac.jp
\end{flushleft}

\end{document}